\documentclass[12pt]{amsart}
\usepackage{amsmath,amssymb,mathabx,amsthm,mathtools}
\usepackage{mathabx}
\usepackage[margin=1.5in]{geometry}
\newtheorem*{thm*}{Theorem}
\newtheorem{thm}{Theorem}

\newtheorem{lem}[thm]{Lemma}

\def\:{\colon}
\def\C{\mathbb{C}}
\def\R{\mathbb{R}}
\def\D{\mathbb{D}}
\def\N{\mathbb{N}}
\def\Z{\mathbb{Z}}
\def\Q{\mathbb{Q}}
\def\bC{\widehat{\mathbb{C}}}
\def\id{\mathrm{id}}

\def\ga{\gamma} 
\def\ra{\rightarrow} 
\def\sub{\subseteq} 
\begin{document}
\title{Canonical embeddings of  pairs of arcs}
\author{Mario Bonk}
\thanks{M.B.\ was partially supported by  NSF grants DMS-1808856 and DMS-2054987.}
\dedicatory{Dedicated to the memory of Walter Hayman}
\address {Department of Mathematics, University of California, Los Angeles,
CA 90095}
\email{mbonk@math.ucla.edu}
\author{Alexandre Eremenko}
\address{Department of Mathematics, Purdue University, West Lafayette, IN 47907}
\email{eremenko@purdue.edu}

\email{}
\date{September 6, 2021}
\begin{abstract} We show that for given four points in the Riemann sphere
and a given isotopy class of two disjoint arcs connecting these points
in two pairs, there exists a unique configuration with the property
that each arc is a hyperbolic geodesic segment
in the complement of the other
arc.
\end{abstract}

\maketitle

In the recent paper \cite{Wang}, Peltola and Wang made a remarkable
restatement of results  in  \cite{EG, EG1} about the existence
and uniqueness of a real rational function with prescribed
real critical points.
To formulate  this precisely, we consider {\em chord diagrams}
in the closure $\overline \D$ of the unit disk $\D =\{z\in \C:|z|<1\}$ in the complex plane $\C$.  Such a chord diagram  has prescribed points  
$a_1,\ldots,a_{2d-2}$, $d\ge 2$, 
on the unit circle and $d-1$ disjoint crosscuts  $e_1, \dots , e_{d-1}$ 
in $\overline \D$ connecting pairs of these points. We  call
such a chord diagram {\em canonical} if every crosscut  $e_k$ is a hyperbolic
geodesic in the unique component of $\D \setminus \bigcup_{j\ne k}e_j$ 
that contains the interior points of  $e_k$.  

Theorem 1.1 in \cite{Wang} states
that for any prescribed points $a_1,\ldots,a_{2d-2}$
there is a unique canonical chord diagram in every combinatorial class.
This canonical chord diagram can be obtained from  the preimage of the real line
under a real rational function of degree $d$ with critical points at
$a_1,\ldots,a_{2d-2}$.
This theorem has important applications for  the study of
the Stochastic Loewner Evolution (SLE).

The number of combinatorial classes of chord diagrams
with prescribed vertices is finite: it is the Catalan number.
In this note we give a simple example of a similar problem
with four prescribed points and infinitely many canonical configurations.  

Our configurations consist
of four distinct points $a_0,a_1, a_2,a_3$ in the Riemann sphere $\bC=\C \cup\{\infty\} $ and two disjoint arcs $\gamma_0$ and $\gamma_1$, where 
$\gamma_0$ has the  endpoints $a_0$  and  $a_1$, and $\gamma_1$ has the 
endpoints $a_2$ and 
$a_3$. We say that two such configurations
are {\em equivalent} if the points are
the same, and the arcs
of the first configuration can be deformed into the arcs
of the second configuration
by an isotopy of the sphere that keeps the endpoints of the arcs fixed.
 A configuration is called {\em canonical}
if for each $k\in\{0,1\}$ the arc 
$\gamma_k$ is a hyperbolic 
geodesic segment in the simply connected
hyperbolic region $\bC\backslash \gamma_{1-k}$.
\begin{thm*} 
For  every equivalence class
of configurations, there exists a unique canonical configuration.
\end{thm*} 
Our proof will show that one can obtain an explicit description
of canonical configurations as follows. We may assume
without loss of generality that
\begin{equation}\label{a}
(a_0,a_1,a_2,a_3)=(\infty,e_1,e_2,e_3), 
\end{equation}
where
$e_1, e_2, e_3\in \C $ and $e_1+e_2+e_3=0$.
Let $\wp$ be the Weierstrass function satisfying
\begin{equation}\label{b} 
(\wp')^2=4(\wp-e_1)(\wp-e_2)(\wp-e_3).
\end{equation}
We denote the line segment joining two points $z,w\in \C$ by $[z,w]$. 
Then each  canonical con\-figuration for the points as in \eqref{a}
has the form 
\begin{equation}\label{c} 
\gamma_0=\wp([0,\omega_1/2]),\quad\gamma_1=
\wp([\omega_2/2, (\omega_1+\omega_2)/2]), 
\end{equation} 
where the pair $\omega_1$ and $\omega_2$
generates  the period lattice of $\wp$, and 
 $e_1=\wp(\omega_1/2)$, $e_2=\wp(\omega_2/2)$, and
$e_3=\wp((\omega_1+\omega_2)/2)$. 

\smallskip In order to prove our theorem, we first state some auxiliary facts. 
An anti-conformal {\em involution} of a region $D$ in the Riemann sphere $\bC$ is an anti-conformal
homeomorphism $\sigma$ of $D$ onto itself such that $\sigma
\circ \sigma=\id_{D}$, where $\id_{D}$ denotes the identity map on $D$.

We will use the following well-known   facts about anti-conformal involutions.

\begin{lem} Let $D$ be a simply connected hyperbolic region in $\bC$, 
and $\sigma$ be an anti-conformal involution
of $D$. Then the set of fixed points of $\sigma$ is a hyperbolic geodesic.
Conversely, for every hyperbolic geodesic there exists a unique anti-conformal
involution of $D$ 
that  fixes all points on this geodesic.
\end{lem}

\begin{proof} By the Riemann mapping theorem, we may assume that $D$
is the unit disk $\D$. If $C$ is a hyperbolic geodesic in 
$\D$, then $C$ is an arc of a circle or a line segment 
that is orthogonal to the unit circle $\partial \D$.
Then reflection  in $C$ is an anti-conformal involution
$\sigma_C$ of $\D$ that fixes every point of $C$. 

Note that if $C=I=(-1,1)$, then $\sigma_I(z)=\bar z$ and so
the anti-conformal involution 
$\sigma_I$ is a hyperbolic isometry.  

Now let $\sigma$ be an arbitrary anti-conformal involution of $\D$.
Then $\sigma$ is also a hyperbolic isometry. Indeed, 
$\tau=\sigma \circ \sigma_I $ is a conformal automorphism
of $\D$ and hence a hyperbolic isometry. This implies that
$\sigma=\tau  \circ \sigma_I $ is a hyperbolic isometry as well.   

This in turn implies that $\sigma$ has a fixed point $w_0\in\D$, namely,
for $w_0$ we can take 
the midpoint of the hyperbolic geodesic segment $S$
joining some point $z_0\in \D $ with $\sigma(z_0)\in \D$.
To see this, note that $\sigma$ is an isometry
on $S=\sigma(S)$ that interchanges the endpoints of $S$. 

By conjugating with an auxiliary automorphism,
we may assume that $w_0=0$. Then $\tau=\sigma\circ\sigma_I $ 
is an automorphism of $\D$ that fixes $0$. Hence 
$\tau(z)=e^{i\theta} z$ with $\theta\in \R$. It follows that 
$\sigma(z)=(\tau\circ\sigma_I)(z)=  e^{i\theta}\overline z,$ 
and so $\sigma$ is equal to the reflection $\sigma_C$ in
the hyperbolic geodesic $C=\{ e^{i\theta/2 }t: -1<t<1\}$. In particular,
$\sigma$ fixes the points in $C$ and no other points.  

The argument also shows that each anti-conformal involution 
of $\sigma$  of $\D$ has the form $\sigma=\sigma_C$ for some
hyperbolic geodesic $C$. This implies that the fixed point set of $\sigma$
uniquely determines $\sigma$.    
\end{proof} 

\begin{lem} An anti-conformal involution $\sigma$
of an annulus  $A=\{ z\in \C :1<|z|<R\}$ with $R>1$  
that  leaves each boundary component  invariant is  of the form
$\sigma(z)=e^{i\theta}\overline{z}$ with  $\theta\in\R$.
\end{lem} 
Note that {\em a priori} the involution $\sigma$
is not defined on the boundary of $A$; so 
by invariance of the boundary components we mean that 
$\sigma(z) \to \partial_k A$ as $z\in A\to \partial_k A $ 
for each boundary component
$\partial_k A=\{z\in \C: |z|=R^k\}$ of $A$ for $k=0,1$. 

\begin{proof} Let $\tau(z)=\overline z$.
Then $\rho=\sigma\circ \tau$ is a conformal
automorphism of $A$ that preserves the boundary components
of $A$. It is well-known that then $\rho(z)=
e^{i\theta}z$ with $\theta\in \R$. Hence 
$\sigma (z)=(\rho\circ \tau)(z)=e^{i\theta}\overline{z}$.  
\end{proof} 

\begin{proof} [Proof of the Theorem.] 
We assume that we have some canonical configuration.
We will analyze the situation and will obtain an
explicit description from which existence and uniqueness will be evident. 

So suppose the disjoint arcs $\gamma_1$ and $\gamma_2$ in $\bC$ 
form a canonical 
configuration. Then by Lemma 1, there exists an  anti-conformal involution
$\sigma_k:\bC\backslash\gamma_{1-k}\to\bC\backslash\gamma_{1-k}$
fixing the points in $\gamma_k$ for $k\in\{0,1\}$.

If we restrict these maps to the ring domain 
$D=\bC\backslash(\gamma_0\cup\gamma_1)$, then
we obtain anti-conformal involutions of $D$ fixing the boundary components.
Now $D$ is conformally equivalent 
to an annulus $A=\{ z:1<|z|<R\}$ with $R>1$.
Then by  Lemma~2,  each  anti-conformal
involution $\sigma_k$ on $D$ corresponds to  a reflection $\tau_k$ in a 
 line
through the origin on $A$.

Conversely, suppose that $\tau_0$ and $\tau_1$ are two reflections
in lines through the origin. Then we can identify or ``weld"
the points on  each   boundary 
component $\partial_k A$  
of $A$ together by using the map $\tau_k$ for $k=0,1$. 
The quotient space carries a natural conformal structure,
and is hence conformally equivalent to the Riemann
sphere by the uniformization theorem. This sphere will carry two distinguished 
arcs $\ga_k$, $k=0,1$,  corresponding to each  boundary  
$\partial_k A$ after the welding. Note that each reflection $\tau_k$ 
passes to  the quotient of $A\cup \partial_k A$
as  an anti-conformal involution fixing the points corresponding to points on $\partial_k A$.
This induces an 
anti-conformal involution of $\bC\setminus \ga_{1-k}$
fixing the points  on $\ga_k$. By Lemma~1, the arc $\ga_k$
is a hyperbolic geodesic segment in the hyperbolic region 
$\bC\setminus \ga_{1-k}$.
It follows that  $\ga_0$ and $\ga_1$ form a canonical configuration. 

We have shown that {\em the canonical configurations are precisely
those that can be obtained from an  annulus 
$ A=\{ z\in \C:1<|z|< R\}$, $R>1$,  by welding
the points in each boundary component $\partial_k A$ together 
by using a reflection $\tau_k$ in a line through the origin for $k=0,1$.} 

Now    a (closed)   annulus $\overline A=\{ z\in \C:1\le |z|\le R\}$   carries an essentially unique (up to scaling) 
flat conformal metric
in which the circles $\{ z\in \C:|z|=R^k\}$, $k=0,1$,  are geodesics. It has length element
\begin{equation}\label{1}
|dz|/|z|. 
\end{equation}
The annulus $\overline A$ equipped with this metric is isometric to the cylinder 
$$\{(x_1,x_2,x_3)\in\R^3:x_1^2+x_2^2=1,\, 0\le x_3\le \log R\}$$ with the Riemannian metric
induced from $\R^3$.

For the proof of the essential uniqueness of such a  metric on $\overline A$, one extends it to $\C^*=\C\setminus\{0\}$ by successive 
reflections and then lifts it to the universal cover $\C$ by the exponential map. The resulting
conformal metric on $\C$ will be complete and flat, and hence 
equal to the Euclidean metric up to a scaling (see Huber
\cite{Huber} for an analytic approach). A representation of the conformal metric on $\overline A$ as in \eqref{1} follows.

After the welding of the boundary components  
$\partial_k A$  by   each reflection $\tau_k$,
the flat metric in \eqref{1} descends to a
flat metric on  the quotient space with possible
singularities in the points of $\ga_0\cup \ga_1$. 
Since $\partial_k A$ is a  geodesic in the flat metric \eqref{1}, 
$\gamma_k$ will be a geodesic arc in this metric with
conic singularities at the endpoints 
and the angles at these singularities are~$\pi$.

So we obtain the following geometric description of canonical configurations:
{\em to each canonical configuration corresponds a flat metric
on the sphere with four conic singularities at $a_0,a_1, a_2, a_3$
with cone angle $\pi$
such that $\gamma_0$ and $\gamma_1$ are geodesic segments.}

The converse is also true: {\em given a flat metric
on the sphere with four conic singularities $a_0,a_1, a_2, a_3$ 
with cone angles $\pi$, 
any pair of disjoint geodesics $\gamma_0$
connecting $a_0$ with $a_1$,  and $\gamma_1$ connecting $a_2$ with $a_3$
is a canonical configuration.}

To see this,  note that we can cut open the sphere along the 
arcs $\gamma_0$ and $\gamma_1$. Then up to scaling,
$D=\bC\backslash(\gamma_0\cup\gamma_1)$
equipped with the flat metric is isometric to an
annulus $A=\{ z\in \C:1<|z|<R\}$ with $R>1$,
equipped with the metric \eqref{1}.  Here each 
geodesic $\ga_k$ is doubled and represented by two circular arcs
$\alpha_k$ and $\alpha'_k$ of equal length that
have common endpoints and whose union is a boundary
component of $A$. We may assume that $\ga_k$ corresponds to  $\partial_k A=
\alpha_k\cup \alpha'_k$ for $k=0,1$.
The sphere $\bC$ with the flat metric and the geodesic 
arcs $\ga_0$ and $\ga_1$ can be recovered from
$A$ if we identify correspond points on $\alpha_k$ and $\alpha'_k$ 
by an isometry fixing the common endpoints
of $\alpha_k$ and $\alpha'_k$.
But such an isometry is necessarily given by a reflection 
$\tau_k$ in a line though the origin.
So we are back to our first description of canonical
configurations  as a quotient space of an annulus $A$. 
  
A  flat metric on the sphere $\bC$ with four prescribed
conic singularities
with angles $\pi$ gives $\bC$ the structure of a parabolic orbifold.  
The corresponding flat metric is unique up to scaling, and obtained by 
pushing the Euclidean metric in the plane forward by 
the universal orbifold covering map
$\Theta\: \C\ra \bC$. For a parabolic orbifold 
with four  conic singularities
with angles $\pi$ this universal orbifold
covering map $\Theta$ is a Weierstrass $\wp$-function
followed by a M\"obius transformation. 
With the normalization \eqref{a},
we actually have $\Theta=\wp$, where $\wp$ is as in \eqref{b}. Then     
the length element of the flat metric is given by $ds=|\wp'(z)||dz|$
and geodesic segments on $\bC$ in the flat metric are
given by images of Euclidean geodesic segments under $\wp$
(for a thorough discussion of the relevant facts about orbifolds see 
\cite[Sections 3.5, A.9, A.10]{BM}). 

For the given normalization \eqref{a}, the arc
$\ga_0$  in a canonical configuration lifts
under $\wp$ to a Euclidean line segment
$[z_0, z_1]\sub \C$ such that $\wp$ is an isometry of 
$[z_0, z_1]$ onto $\ga_0$. Here we may assume that 
$\wp(z_0)=\infty$ and $\wp(z_1)=e_1$.

Let $\Gamma\sub \C$ be the period lattice of $\wp$. 
Since we have translation invariance of
$\wp$ under $\Gamma$  and $\wp^{-1}(\infty)=\Gamma$,
we may further assume that $z_0=0$.  Since $\wp(z_1)=e_1$,
the point $z_1$ must be a half-period of $\wp$,
i.e., $z_1\in \frac12 \Gamma$.
Now 
$\wp$ is injective on $[z_0,z_1]=[0,z_1]$,
and so the point $z_1$ must be of the form  $z_1=\omega_1/2$, where 
$\omega_1\ne 0$ is a {\em primitive} element of $\Gamma$, i.e., 
$\omega_1$ cannot be represented in the form $\omega_1=n \ga$ 
with $n\in \N$, $n\ge 2$, and $\ga\in \Gamma$. It follows that 
$\ga_0=\wp([0,\omega_1/2)]$ as in  \eqref{c}. 

Since $\omega_1$ is a primitive element of $\Gamma$, there exists 
an element $\omega_2\in \Gamma$ such that $\omega_1$ and $\omega_2$ 
form  a basis of $\Gamma$. For given $\omega_1$, the choice of 
$\omega_2$ is not unique, but if we
make one choice for $\omega_2$,
then all other choices $\omega_2'$ are of the form 
\begin{equation}\label{gen}
\omega_2'=\pm \omega_2+ n\omega_1
\end{equation} 
with $n\in \Z$.
Note that $\wp$ maps the half-periods $\frac 12\omega_2'$  
to $e_2$ or $e_3$. With suitable choice of $\omega_2$ we may
assume that $\wp(\omega_2/2)=e_2$. 

We now lift the second arc $\ga_1$ in our canonical
configuration under $\wp$ to a 
line segment  $[z_2, z_3]\sub \C$ starting at $z_2=\omega_2/2$. 
Here $\wp(z_3)=e_3$, and so $z_3\in \frac12 \Gamma$ is a half-period.  

We can say more here. Since the $\wp$-function satisfies 
\begin{equation}\label{wp}
 \wp(\pm z +\alpha )=\wp(z) \text { for $z\in \C$ and $\alpha\in \Gamma$},
\end{equation}
it is invariant under reflections in half-periods. 
This implies that  $\wp^{-1} (\ga_0)$ contains the full line 
passing through $z_0=0$ and $z_1=\omega_1/2$. Similarly, 
$\wp^{-1} (\ga_1)$ contains the full line  passing through $z_2
=\omega_2/2$ and the half-period $z_3$.
Since $\ga_0$ and $\ga_1$ are disjoint,
these lines cannot meet and hence must be parallel.
It follows that $z_3$ necessarily has the form 
$z_3=\tfrac12(n\omega_1 +\omega_2)$ with $n\in \Z$. Since 
$\wp(z_3)=e_3$, the integer $n$ must be odd, and since 
$\wp$ is injective on the lift $[z_2,z_3]$, we must have 
$z_3=  \tfrac12(\pm  \omega_1 +\omega_2)$.
By reflection symmetry of $\wp$ in the half-period
$z_2=\omega_2/2$ and replacing 
the original lift $[z_2, z_3]$ by its reflection
image in $z_2$, we may assume that $z_3=\tfrac12(\omega_1+\omega_2)$. 
  
We conclude that under the normalization \eqref{a},
arcs in a canonical configuration $\ga_0$ and $\ga_1$ have the form \eqref{c}.
Conversely, arcs as in \eqref{c} form a canonical
configuration as follows from  our  geometric description of 
canonical configuration in terms of flat metrics on $\bC$. 

It remains to show that under the assumption \eqref{a}
we obtain exactly one canonical configuration of arcs in each isotopy class. 

Our previous analysis shows that in \eqref{c} the arcs 
$\ga_0$ and $\ga_1$ are uniquely determined once we know 
the primitive element $\omega_1$ up to sign. Indeed, it is clear that this
determines $\ga_0$. Moreover, the choice of $\omega_1$ up to sign
does not determine  $\omega_2$ uniquely, but it follows from 
\eqref{gen} and \eqref{wp} that the arc $\ga_2=[\wp(\omega_2/2, 
(\omega_1+\omega_2)/2]$ is uniquely determined
independent of the choice of
the sign of $\omega_1$ and the choice of $\omega_2$.

Now suppose we have chosen a fixed basis 
$\omega^0_1$ and $\omega^0_2$ of the period lattice $\Gamma$.
Then the primitive elements $\omega_1$ of $\Gamma$ are precisely
the elements of the form 
$$\omega_1=r\omega^0_1+s\omega^0_2, $$ where $r,s\in \Z$ are relatively prime.
By choosing the appropriate sign of $\omega_1$,
we may assume that $s\ge 0$ and that $r=1$ if $s=0$. 
The ratio $r/s\in \widehat\Q=\Q\cup\{\infty\}$ describes the slope 
of the line in $\C$ passing through $0$ and $\omega_1$
and we have a bijective correspondence 
between slopes $r/s\in  \widehat\Q$ and primitive
elements $\pm \omega_1$ of $\Gamma$ up to sign. 
 
Now it is a well known fact that isotopy classes of arcs in a sphere 
wit h four marked points  are in one-to-one correspondence 
with these rational slopes $r/s\in  \widehat\Q$ (see 
\cite[Chapter 2]{FM} for a related discussion). 
In our case, 
$\pm \omega_1\leftrightarrow \wp( [0, \omega_1/2])$
induces a bijective correspondence between primitive
elements $\pm \omega_1$ of $\Gamma$ up to sign and 
isotopy classes of arcs $\ga_0$ in $\bC$ with  marked points as
\eqref{a}   
(see \cite[Section 2.6] {BHI} for a thorough discussion in
the spirit of the present considerations). Note that the isotopy class of 
$\ga_0$ uniquely determines the isotopy class of the pair 
$(\ga_0, \ga_1).$ The statement follows. 
\end{proof} 

\bigskip 
\noindent 
{\bf Acknowledgment.} We thank Daniel Meyer for his remarks on this paper.

\end{document}